\documentclass[11pt]{amsart}
\usepackage{amsfonts}
\usepackage{amsmath}
\usepackage{amssymb}
\usepackage{color}
\usepackage{graphicx}
\usepackage{xspace}
\usepackage{hyperref}
\usepackage{axodraw4j}
\usepackage[all]{xy}
 \font \eightrm=cmr8

 \newcommand{\nc}{\newcommand}

\newtheorem{thm}{Theorem}

\newtheorem{cor}[thm]{Corollary}

\newtheorem{prop}[thm]{Proposition}
\newtheorem{defn}{Definition}
\newtheorem{rmk}[thm]{Remark}

\def\bcdot{\scalebox{2}{\hbox{$.$}}}
\nc{\rhdi}{\overline{\rhd}}
\nc{\lgl}{[\![}
\nc{\rgl}{]\!]}

\nc{\ignore}[1]{{}}
\nc{\mrm}[1]{{\rm #1}}
\nc{\dirlim}{\displaystyle{\lim_{\longrightarrow}}\,}
\nc{\invlim}{\displaystyle{\lim_{\longleftarrow}}\,}
\nc{\vep}{\varepsilon} \nc{\ep}{\epsilon}
\nc{\sigmat}{\widetilde\sigma}
\nc{\ostar}{\overline{*}}

\nc{\mchar}{\mrm{Char}}
\nc{\Hom}{\mrm{Hom}}
\nc{\id}{\mrm{id}}

\nc{\remark}{\noindent{\bf{Remark:}}}
\nc{\remarks}{\noindent{\bf{Remarks:}}}

 \nc{\delete}[1]{}
 \nc{\grad}[1]{^{({#1})}}
 \nc{\fil}[1]{_{#1}}

\nc{\BA}{{\Bbb A}} \nc{\CC}{{\Bbb C}} \nc{\DD}{{\Bbb D}}
\nc{\EE}{{\Bbb E}} \nc{\FF}{{\Bbb F}} \nc{\GG}{{\Bbb G}}
\nc{\HH}{{\Bbb H}} \nc{\LL}{{\Bbb L}} \nc{\NN}{{\Bbb N}}
\nc{\PP}{{\Bbb P}} \nc{\QQ}{{\Bbb Q}} \nc{\RR}{{\Bbb R}}
\nc{\TT}{{\Bbb T}} \nc{\VV}{{\Bbb V}} \nc{\ZZ}{{\Bbb Z}}
\nc{\Cal}[1]{{\mathcal {#1}}}
\nc{\mop}[1]{\mathop{\hbox {\rm #1} }\nolimits}
\nc{\smop}[1]{\mathop{\hbox {\eightrm #1} }\nolimits}
\nc{\mopl}[1]{\mathop{\hbox {\rm #1} }\limits}
\nc{\frakg}{{\frak g}}
\nc{\g}[1]{{\frak {#1}}}
\def \restr#1{\mathstrut_{\textstyle |}\raise-8pt\hbox{$\scriptstyle #1$}}
\def \srestr#1{\mathstrut_{\scriptstyle |}\hbox to
  -1.5pt{}\raise-4pt\hbox{$\scriptscriptstyle #1$}}
\nc{\wt}{\widetilde}
\nc{\wh}{\widehat}
\nc{\un}{\hbox{\bf 1}}
\nc{\redtext}[1]{\textcolor{red}{\tt #1}}
\nc{\bluetext}[1]{\textcolor{blue}{#1}}
\nc{\comment}[1]{[[{\tt {#1}}]] }
\nc{\R}{{\mathbb R}}

\nc\fleche[1]{\mathop{\hbox to #1 mm{\rightarrowfill}}\limits}
\def\semi{\mathrel{\times}\kern -.85pt\joinrel\mathrel{\raise 1.4pt\hbox{${\scriptscriptstyle |}$}}}

\def\ta1{{\scalebox{0.2}{ 
\begin{picture}(12,12)(38,-38)
\SetWidth{0.5} \SetColor{Black} \Vertex(45,-33){5.66}
\end{picture}}}}

\begin{document}

\title[Post-groups]
      {Free post-groups, post-groups from group actions,\\ and post-Lie algebras}
                  
\author[M.~J.~H.~Al-Kaabi]{Mahdi Jasim Hasan Al-Kaabi}
\address{Mathematics Department, 
		College of Science, Mustansiriyah University, Palestine Street, 
		P.O.Box 14022, 
		Baghdad, IRAQ.}       
         \email{mahdi.alkaabi@uomustansiriyah.ed.iq}
      
\author[K.~Ebrahimi-Fard]{Kurusch Ebrahimi-Fard}
\address{Department of Mathematical Sciences, 
		NTNU, NO 7491 Trondheim, Norway. 
		Centre for Advanced Study CAS, Drammensveien 78, 
		0271 Oslo, Norway.}
		\email{kurusch.ebrahimi-fard@ntnu.no}
		\urladdr{https://folk.ntnu.no/kurusche/}
         
\author[D.~Manchon]{Dominique Manchon}
\address{LMBP,
         C.N.R.S.-Universit\'e Clermont-Auvergne,
         3 place Vasar\'ely, CS 60026,
         63178 Aubi\`ere, France}       
         \email{Dominique.Manchon@uca.fr}
         \urladdr{https://lmbp.uca.fr/~manchon/}

\date{\today}

\begin{abstract}
After providing a short review on the recently introduced notion of post-group by Bai, Guo, Sheng and Tang, we exhibit post-group counterparts of important post-Lie algebras in the literature, including the infinite-dimensional post-Lie algebra of Lie group integrators. The notion of free post-group is examined, and a group isomorphism between the two group structures associated to a free post-group is explicitly constructed.
\end{abstract}

\maketitle


\tableofcontents


\section{Introduction}
\label{sect:intro}

The notions of pre- and post-group have been put forward recently by Chengming Bai, Li Guo, Yunhe Sheng and Rong Tang in the work \cite{BGST2023}. Both motivation and terminology derive from the corresponding infinitesimal objects known as pre- and post-Lie algebras. The former simultaneously appeared sixty years ago in the works of Murray Gerstenhaber \cite{G1963} and Ernst Vinberg \cite{V1963}. The idea can be traced even further back to work by Arthur Cayley \cite{C1857} on the connection between trees and vector fields. The notion of post-Lie algebra, on the other hand, appeared far more recently, in works by Bruno Vallette in 2007 \cite{BV2007} and, independently, by Hans Munthe-Kaas and Will Wright in 2008 \cite{MKW2008} who introduced the closely related notion of $D$-algebra and highlighted their relevance in the context of Klein geometries. In fact, pre- as well as post-Lie algebras arise naturally in the context of the geometry of invariant connections defined on manifolds. Examples of pre- and post-groups are easily identified in the works by Daniel Guin and Jean-Michel Oudom \cite{GO2008} respectively by Alexander Lundervold, Hans Munthe--Kaas and the second author~\cite{LEFMK2015}.

\medskip

A post-group is a group $(G,.)$ together with a family $(L_a^\rhd)_{a \in G}$ of group automorphisms such that the relation
$$
	L^\rhd_a \circ L^\rhd_b=L^\rhd_{a . L^\rhd_a(b)}
$$
holds for any $a,b \in G$. The unit of $(G,.)$ is denoted by $e$. The triangle $\rhd$ denotes the additional binary product $a \rhd b =: L_a^\rhd (b)$ defined on $G$. A pre-group is a post-group in which the group law $.$ is commutative. The definition of a post-group formalises properties of group-like elements in the completion of the enveloping algebra of a post-Lie algebra (in a pre-Lie algebra for a pre-group) \cite{GO2008,LEFMK2015}. Furthermore, one observes that the notion of post-group is equivalent to two notions closely related to set-theoretical solutions of the Yang--Baxter equation \cite{ESS99}, namely braided groups \cite{LYZ2000} and skew-braces \cite{GV2017}, see \cite[Section 3]{BGST2023}.  The defining axioms of a post-group have been relaxed by S.~Wang in \cite{W2023}, leading to the notion of weak twisted post-group.\\

We begin this short note with a quick review of the basic properties of post-groups in Section \ref{sect:ib}. A short account of the main results of \cite{BGST2023} is provided in Section \ref{review}. In Section \ref{sect:weak}, we give a short account of S.~Wang's weak twisted post-groups \cite{W2023}. Three new nontrivial examples of weak post-groups (in the sense of S.~Wang, but with the supplementary property that $L^\rhd_e$ is bijective\footnote{which immediately implies $L^\rhd_e=\smop{Id}_G$, see proof of Proposition \ref{wtpg} together with \eqref{action-unit}.}, and without twist) are discussed in Section \ref{examples}. The first example is given by maps from a set $M$ into a group $G$, with the latter acting on $M$ from the right. The second example is the smooth analogue of the first one, namely smooth maps from a smooth manifold $\mathcal{M}$ into a Lie group $G$ right-acting differentiably on $\mathcal{M}$. In Paragraph \ref{MKW-PG}), we show that the post-Lie algebra naturally associated with this weak post-Lie group coincides with the post-Lie algebra of Lie group integrators considered in \cite{HA13}.\\

The last family of examples provides post-groups in the strict sense (Theorem \ref{main-free}), namely free post-groups generated by left-regular diagonal magmas. A magma $(M,\rhd)$ is left-regular whenever the left multiplication operators $L_x^\rhd=x\rhd -$ are bijective. A left-regular magma is called \textsl{diagonal} in this work when the application $x\mapsto (L_x^\rhd)^{-1}(x)$ is moreover bijective from $M$ onto $M$. Our two main results in this last section can be summarized as follows:
\begin{itemize}
\item We prove (Theorem \ref{main-free}) that the free group $F_M$ generated by $M$ is a post-group, and that it is free in the following sense: for any post-group $(G,.,\rhd)$, and for any magma morphism $\varphi:(M,\rhd)\to (G,\rhd)$, the unique group morphism $\Phi:F_M\to G$ extending $\varphi$ is a morphism of post-groups. 

\item We construct (Proposition \ref{iso}) a group isomorphism $\mathcal K:(F_M,*)\to(F_M,.)$ between the two group structures of a free post-group, reminiscent of Aleksei V.~Gavrilov's $K$-map between the two Lie algebra structures of a free post-Lie algebra \cite{Gavrilov2007, Gavrilov2012}, see also \cite{AEMM2022,Foissy2018}.\\
\end{itemize}

\noindent \textbf{Acknowledgements :} We thank Pierre Catoire for his insightful remarks on a previous version of this work. M.~J.~H.~Al-Kaabi is funded by the Iraqi Ministry of Higher Education and Scientific Research. He also thanks Mustansiriyah University, College of Science, Mathematics Department for support. K.~Ebrahimi-Fard is supported by the Research Council of Norway through project 302831 “Computational Dynamics and Stochastics on Manifolds” (CODYSMA). He also thanks the Centre for Advanced Study (CAS) in Oslo for support. D.~Manchon acknowledges a support from the grant ANR-20-CE40-0007 \textsl{Combinatoire Alg\'ebrique, Renormalisation, Probabilit\'es Libres et Op\'erades}.

\section{Basic definitions and properties}
\label{sect:ib}

\begin{defn} \cite{BGST2023}
A \textbf{post-group} $(G,.,\rhd)$ is a group $(G,.)$ endowed with a binary map $\rhd: G\times G\to G$ such that $L_a^\rhd (-):=a\rhd -$ is a group automorphism, and such that the following identity holds:
\begin{equation}
\label{GL}
	(a*b)\rhd c=a\rhd(b\rhd c),
\end{equation}
for any $a,b,c\in G$, where 
\begin{equation}
\label{GLgroup}
	a*b := a.(a\rhd b).
\end{equation}
\end{defn}

\noindent The inverse of any $a\in G$ will be denoted by $a^{.-1}$.  The following statement identifies the product \eqref{GLgroup} as a group law.

\begin{prop}\label{groupGL}\cite[Theorem 2.4]{BGST2023}
The binary map $*:G\times G\to G$ defined by $a*b:=a.(a\rhd b)$ provides $G$ with a second group structure. Both groups $(G,.)$ and $(G,*)$ share the same unit $e$, and the inverse for $*$ is given by
\begin{equation}\label{inverse-GL}
	a^{*-1}=(L_a^\rhd)^{-1}a^{.-1}.
\end{equation}
The binary map $*$ is an action of the group $(G,*)$ on the group $(G,.)$ by automorphisms.
\end{prop}

\begin{proof}
The proof can be found in \cite{BGST2023}, we reproduce it for the reader's convenience: associativity of $*$ comes from the associativity of the group law on $G$ by following direct computation:
\begin{eqnarray*}
	(a*b)*c
	&=&\big(a.(a\rhd b)\big)*c\\
	&=&\big(a.(a\rhd b)\big).\Big(\big(a.(a\rhd b)\big)\rhd c\Big)\\
	&=&\big(a.(a\rhd b)\big).\big(a\rhd(b\rhd c)\big)
\end{eqnarray*}
whereas
\begin{eqnarray*}
	a*(b*c)
	&=&a*\big(b.(b\rhd c)\big)\\
	&=&a.\Big(a\rhd\big(b.(b\rhd c)\big)\Big)\\
	&=&a.\Big((a\rhd b).\big(a\rhd(b\rhd c)\big)\Big).
\end{eqnarray*}
Let $e$ be the unit of the group $(G,.)$. From $a\rhd e=e$ for any $a\in G$ we easily deduce that 
$$
	a*e=a.(a\rhd e)=a.
$$ 
On the other hand, we have $e\rhd(e\rhd a)=\big(e.(e\rhd e)\big)\rhd a=e\rhd a$, in other words $(L_e^\rhd)^2=L_e^\rhd$. From the bijectivity of $L_e^\rhd$ we deduce $L_e^\rhd=\mop{Id}_G$, hence
\begin{equation}
\label{action-unit}
	e\rhd a=a.
\end{equation}
This implies $e*a=e.(e\rhd a)=a$ and $a*e=a.(a\rhd e)=a$, hence $e$ is the unit for the new product $*$. Finally, defining the element $b:=(L_a^\rhd)^{-1}a^{.-1}$, for $a \in G$, we check quickly that
\begin{eqnarray*}
	a*b
	&=&a.(a\rhd b)\\
	&=&a.\big(L_a^\rhd(L_a^\rhd)^{-1}a^{.-1}\big)\\
	&=&a.a^{.-1}=e.
\end{eqnarray*}
Any element in $G$ therefore admits a right-inverse with respect to the product $*$. A standard argument shows that the right-inverse is also a left-inverse, and that the inverse thus obtained is unique. 
\end{proof}

\begin{defn}\label{def:gl}
Let $(G,.,\rhd)$ be a post-group. The associated \textbf{Grossman--Larson group}\footnote{It is called the subjacent group in \cite{BGST2023}. It makes however sense to take into account that any post-group $(G,.,\rhd)$ gives rise to two subjacent groups, namely $(G,.)$ and $(G,*)$. The terminology we have chosen refers to the Grossman--Larson bracket in a post-Lie algebra \cite{ELM2015}, which itself refers to the Grossman--Larson product of rooted forests \cite{GL1989, MKW2008}.} $(G,*)$ has the group law $a*b=a.(a\rhd b)$, for any $a,b \in G$.
\end{defn}

\begin{defn}\cite[Definition 2.5]{BGST2023}
A \textbf{pre-group} is a post-group $(G,.,\rhd)$ where the group $(G,.)$ is Abelian.
\end{defn}

This terminology should remind the reader of the well-known fact that a pre-Lie algebra can be seen as a post-Lie algebra with an Abelian Lie bracket. Grossman--Larson groups associated to pre-groups are not Abelian in general. Pre-groups naturally associated with free pre-Lie algebras appear as early as 1981 under the terminology ``group of formal flows'' \cite{AG1981, Manchon2011, S2022}.

\begin{defn}
Let $(G,.,\rhd)$ be a post-group. The \textbf{opposite post-group} is given by $(G,\bcdot,\blacktriangleright)$ with $a\bcdot b:=b.a$ and $a \blacktriangleright b:=a.(a\rhd b).a^{.-1}$, for any $a,b\in G$.
\end{defn}

\noindent We leave it to the reader to check that $(G,\bcdot,\blacktriangleright)$ is indeed a post-group, sharing with $(G,.,\rhd)$ the same Grossman--Larson product, namely
\begin{equation*}
	a*b
	=a.(a\rhd b)
	=a\bcdot(a\blacktriangleright b).
\end{equation*}

\section{Quick review of the main results of reference \cite{BGST2023}}\label{review}

We offer here a quick guided tour of the article \cite{BGST2023} by C.~Bai, L.~Guo, Y.~Sheng and R.~Tang, where the reader will find the detailed statements and proofs.

\subsection{Post-groups, braided groups and skew-braces}

The main structural result of \cite{BGST2023} is the equivalence between three different notions, i.e., that of a post-group, a braided group and a skew brace (Section 3 therein). Braided groups appeared in 2000 in an influential article on set-theoretical solutions of the Yang--Baxter equation by J.-H.~Lu, M.~Yan and Y.-C.~Zhu \cite{LYZ2000}. Another construction related to Yang--Baxter equations, namely skew-braces, appeared more recently in the 2017 article by L.~Guarnieri and L.~Vendramin \cite{GV2017}, following the introduction of braces by W.~Rump\footnote{These braces should not be confused with their homonyms related to pre-Lie algebras and operads, see e.g.~\cite{C2002, GO2008,MQS2020}. Interestingly enough, this conflict of terminology identifies pre-Lie algebras with strongly nilpotent braces in the first sense \cite{S2022}, and with symmetric braces in the second sense \cite{GO2008}.} \cite{R2007}.\\

Let $\mathcal C$ be a \textbf{braided monoidal set category}, i.e.~a full subcategory of the category of sets, stable by cartesian product and endowed with a braiding $\sigma$, that is, a collection of bijective maps $\sigma_{XY}: X \times Y \to Y \times X$ indexed by ordered pairs $(X,Y)$ of objects of $\mathcal C$, subject to 
\begin{itemize}
\item functoriality: for any pair of set maps $f:X\to X'$ and $g:Y\to Y'$, the equality
$$
	\sigma_{X' Y'}\circ (f\times g)=(g\times f)\circ \sigma_{XY}
$$
holds.
\item compatibility with the cartesian product:
\begin{equation}
\label{double-tresse}
	\sigma_{X\times X',\,Y\times Y'}
	=(\mop{Id}_Y\times\sigma_{XY'}\times \mop{Id}_{X'})\circ(\sigma_{XY}\times\sigma_{X'Y'})
	\circ (\mop{Id}_X\times\sigma_{X'Y}\times \mop{Id}_{Y'}),
\end{equation}
\item the hexagon equation
\begin{equation}
\label{hexagone}
	(\sigma_{YZ}\times \mop{Id}_X)\circ (\mop{Id}_Y\times  \sigma_{XZ})\circ (\sigma_{XY}\times \mop{Id}_Z)
	=(\mop{Id}_Z\times  \sigma_{XY})\circ (\sigma_{XZ}\times \mop{Id}_Y)\circ (\mop{Id}_X\times  \sigma_{YZ}).
\end{equation}
\end{itemize}
In particular, any object $S$ of $\mathcal C$ is a braided set, the braiding map $\sigma_{SS}: S \times S \to S \times S$ being a bijection satisfying the braid equation
\begin{equation}
\label{braid}
	(\sigma\times \mop{Id}_S)\circ (\mop{Id}_S\times  \sigma)\circ (\sigma\times \mop{Id}_S)
	=(\mop{Id}_S\times  \sigma)\circ (\sigma\times \mop{Id}_S)\circ (\mop{Id}_S\times  \sigma),
\end{equation}
where we have abbreviated $\sigma_{SS}$ by $\sigma$. The category of sets itself is braided with the flip maps $P_{XY}:X\times Y\to Y\times X$ defined by $P_{XY}(a,b)=(b,a)$ (trivial braiding). A map $\sigma:S\times S\to S\times S$ verifies \eqref{braid} if and only if the map $R=P_{SS}\circ\sigma$ is a solution of the set-theoretical Yang--Baxter equation\footnote{The term "Yang--Baxter equation" is sometimes used for the braid equation \eqref{braid} in the literature, thus bringing some confusion. We adopt here the conventions of \cite{LYZ2000}.}
\begin{equation}
\label{yb}
	R_{12}\circ\ R_{13}\circ R_{23} = R_{23}\circ\ R_{13}\circ R_{12},
\end{equation}
where, as usual, $R_{12}= R \times \mop{Id}_S$, $R_{23}=\mop{Id}_S\times R$, and $R_{13}=(\mop{Id}_S\times\tau)\circ(R\times \mop{Id}_S)\circ (\mop{Id}_S\times\tau)$. The trivial braiding satisfies the additional symmetry property $P_{YX}\circ P_{XY}=\mop{Id}_{X\times Y}$ for any sets $X,Y$.\\

\begin{defn}\label{BG}\cite{LYZ2000}
A \textbf{braided group} is a commutative group in the braided monoidal set category generated by it. To be concrete, it is a pair $(G,\sigma)$ where $G$ is a group and $\sigma: G \times G \to G \times G$ is a bijection such that
\begin{itemize}
\item $\sigma\circ(m\times m)=(m\times m)\circ \wt\sigma$, where $ \wt\sigma=\sigma_{G\times G,\,G\times G}:G^4\to G^4$ is given by \eqref{double-tresse},\label{BG-one} and where $m:G\times G\to G$ is the group multiplication,
\item $m\circ\sigma=m$.\label{BG-two}
\end{itemize}
\end{defn}

\noindent In a braided group $(G,\sigma)$, it turns out that \cite[Theorem 1 \& Theorem 2]{LYZ2000}
\begin{itemize}
\item the map $\sigma$ in a braided group verifies the braid equation \eqref{braid},
\item the two maps $\rightharpoonup,\leftharpoonup:G\times G\to G$ defined by
\begin{equation}
\label{actions}
	\sigma(g,h)=(g\rightharpoonup h,\, g\leftharpoonup h)
\end{equation}
are respectively a left action and a right action of the group $G$ on itself. Conversely, if a left action $\rightharpoonup$ and a right action $\leftharpoonup$ together fulfil the compatibility condition $gh=(g\rightharpoonup h).(g\leftharpoonup h)$ for any $g,h\in G$, then $(G,\sigma)$ is a braided group, with the braiding $\sigma$ given by \eqref{actions}.
\end{itemize}

\begin{thm}\cite[Proposition 3.13 \& Proposition 3.17]{BGST2023}
For any post-group $(G,.,\rhd)$, the Grossman--Larson group $(G,*)$ is a braided group, with braiding $\sigma:G\times G\to G\times G$ given by
\begin{equation}
\label{post2braided}
	\sigma(g,h)
	:=\big(g\rhd h,\,(g\rhd h)^{*-1}*g*h\big).
\end{equation}
Conversely, any braided group $(G,\sigma)$ with product denoted by $*$ gives rise to a post-group $(G,.,\rhd)$ with
\begin{equation}
\label{braided2post}
	g\rhd h:=g\rightharpoonup h \hbox{ and } g.h
	:=g*(g^{*-1}\rightharpoonup h),
\end{equation}
where we have used the notation $\sigma(g,h):=(g\rightharpoonup h,\,g \leftharpoonup h)$. Both correspondences are mutually inverse.
\end{thm}

\begin{prop}\label{braidop}
Let $(G,.,\rhd)$ be a post-group, and let $\sigma:G\times G\to G\times G$ the corresponding braiding given by \eqref{post2braided}. The braiding corresponding to the opposite post-group $(G,\bcdot,\blacktriangleright)$ is $\sigma^{-1}$.
\end{prop}

\begin{proof}
It is immediate from Definition \ref{BG} that, for any braided group $(G,\sigma)$, the pair $(G,\sigma^{-1})$ is also a braided group. Now let $(G,.,\rhd)$ be a post-group, let $\sigma$ be the corresponding braiding map, and let $\sigma'$ be the braiding map of the opposite post-group $(G,\bcdot,\blacktriangleright)$. From \eqref{post2braided}, we have $\sigma(g,h)=(H,G)$ with $H=g\rhd h$ and $G=(g\rhd h)^{*-1}*g*h$. Now we have
\begin{eqnarray*}
	\sigma'\circ\sigma(g,h)
	&=&\sigma'(H,G)\\
	&=&\big(H\blacktriangleright G, (H\blacktriangleright G)^{*-1}*H*G\big)\\
	&=&\Big(H.(H\rhd G).H^{.-1},\,\big(H.(H\rhd G).H^{.-1}\big)^{*-1}*H*G\Big)\\
	&=&(K,K^{*-1}*H*G),
\end{eqnarray*}
with $K:=H.(H\rhd G).H^{.-1}$. We therefore have
\begin{eqnarray*}
	K
	&=&(g\rhd h).\Big((g\rhd h)\rhd\big((g\rhd h)^{*-1}*g*h\big)\Big).(g\rhd h)^{.-1}\\
	&=&(g\rhd h).\Big\{(g\rhd h)\rhd\Big((g\rhd h)^{*-1}*\big(g.(g\rhd h)\big)\Big)\Big\}.(g\rhd h)^{.-1}\\
	&=&(g\rhd h).\bigg\{(g\rhd h)\rhd\Big\{(g\rhd h)^{*-1}.\Big((g\rhd h)^{*-1}\rhd\big(g.(g\rhd h)\big)\Big)\Big\}\bigg\}.(g\rhd h)^{.-1}\\
	&=&(g\rhd h).\big((g\rhd h)\rhd (g\rhd h)^{*-1}\big).\Big\{(g\rhd h)\rhd\Big((g\rhd h)^{*-1}\rhd\big(g.(g\rhd h)\Big)\Big\}.(g\rhd h)^{.-1}\\
	&=&(g\rhd h).\big((g\rhd h)\rhd (g\rhd h)^{*-1}\big).\big(g.(g\rhd h)\big).(g\rhd h)^{.-1}\\
	&=&(g\rhd h).(g\rhd h)^{.-1}.g\\
	&=&g,
\end{eqnarray*}
so that $K^{*-1}*H*G=g^{*-1}*(g\rhd h)*(g\rhd h)^{*-1}*g*h=h$. Therefore we have $\sigma' \circ \sigma(g,h) = (g,h)$. Both bijective maps $\sigma$ and $\sigma'$ are mutually inverse, which proves Proposition \ref{braidop}.
\end{proof}

\begin{cor}\cite[Remark 3.15]{BGST2023}
The braiding map corresponding to a pre-group is involutive.
\end{cor}

\begin{proof}
This is a direct consequence of the fact that any pre-group is equal to its opposite.
\end{proof}

\begin{defn}\cite{BGST2023,GV2017}
A \textbf{skew-left brace} is a set $G$ endowed with two group structures $(G,.)$ and $(G,*)$ such that
\begin{equation}
\label{sb}
	g*(h.k)=(g*h).g^{.-1}.(g*k)
\end{equation}
for any $g,h,k\in G$.
\end{defn}

Any post-group $(G,.,\rhd)$ gives rise to the left-skew brace $(G,.,*)$ where $*$ is the Grossman--Larson product \cite[Proposition 3.22]{BGST2023}. Conversely, any left skew-brace gives rise to the post-group $(G,.,\rhd)$ with $g\rhd h:=g^{.-1}.(g*h)$ for any $g,h\in G$. Both correspondences are mutually inverse \cite[Proposition 3.24]{BGST2023}. To complete the picture, the opposite of a skew-brace $(G,.,*)$ will be defined as the skew-brace $(G,\bcdot,*)$ where, as above, $a\bcdot b:=b.a$ for any $a,b\in G$. A left skew-brace where the dot-group product is commutative is a left brace: see \cite{R2007}, or \cite{S2022} for a recent, purely algebraic account.\\

We remark that the recent notions of both left-skew brace and post-group, therefore appear to be reformulations of the older notion of braided group developed by J.-H.~Lu, M.~Yan and Y.~Zhu in \cite{LYZ2000}, shedding new light to it.

\subsection{Post-Lie groups and post-Lie algebras}

\begin{defn}\cite[Definition 4.2]{BGST2023}
A \textbf{post-Lie group} is a post-group $(G,.,\rhd)$ where $G$ is a smooth manifold, and where both operations are smooth maps from $G\times G$ to $G$.
\end{defn}

\begin{defn}\cite{BV2007} A post-Lie algebra (over some field $\mathbf k$) is a triple $(\mathfrak g, [-,-],\rhd)$ where $(\mathfrak g, [-,-])$ is a Lie algebra, and $\rhd: \mathfrak g \times \mathfrak g$ is a bilinear map such that, for any $X,Y,Z \in \mathfrak g$,
\begin{itemize}
	\item $X\rhd[Y,Z]=[X\rhd Y,Z]+[Y,X\rhd Z]$,
	\item $[X,Y]\rhd Z=X\rhd(Y\rhd Z)-(X\rhd Y)\rhd Z-Y\rhd(X\rhd Z)+(Y\rhd X)\rhd Z$.
\end{itemize}
\end{defn}

As key property of post-Lie algebra we recall from \cite{ELM2015} that $(\mathfrak g,\lgl-,-\rgl)$ is the corresponding \textbf{Grossman--Larson Lie algebra}\footnote{called subjacent Lie algebra in \cite{BGST2023}.} given by the Grossman--Larson bracket defined by
\begin{equation}
\label{GLLiebracket}
	\lgl X,Y \rgl:=[X,Y]+X\rhd Y-Y\rhd X,
\end{equation}
for any $X,Y\in\mathfrak g$. The older notion of pre-Lie algebra \cite{G1963,V1963} follows from that of post-Lie algebra in the case of the latter having an Abelian Lie bracket. The corresponding Grossman--Larson bracket is thus given by the anti-symmetrization of the pre-Lie product, that is, pre-Lie algebras are Lie admissible. We remark that in \cite{AG1981} pre-Lie algebras are called chronological algebras. See \cite{Burde2006,Manchon2011} for detailed reviews.\\

\begin{defn}
Let $(\mathfrak g,[-,-],\rhd)$ be a post-Lie algebra. Defining $[X,Y]^{\smop{op}}:=-[X,Y]$ and $X\blacktriangleright Y:=X\rhd Y+[X,Y]$, we have that $(\mathfrak g, [-,-]^{\smop{op}},\blacktriangleright)$ is a post-Lie algebra, which we call the \textbf{opposite post-Lie algebra}.
\end{defn}

\noindent Checking the post-Lie algebra axioms for  $(\mathfrak g, [-,-]^{\smop{op}},\blacktriangleright)$ is left to the reader. Both post-Lie algebras share the same Grossman--Larson bracket \cite[Paragraph 2.1]{AEMM2022}.\\

Section 4 of \cite{BGST2023} elucidates the relationship between post-Lie groups and post-Lie algebras. Namely, the Lie algebra $\mathfrak g$ of a post-Lie group $G$ is a post-Lie algebra. The bilinear map $\rhd$ is given by
\begin{equation}
\label{diff}
	X\rhd Y:=\frac{d}{dt}\restr{t=0}\frac{d}{ds}\restr{s=0}\exp(tX)\rhd\exp(sY)
\end{equation}
(see \cite[Theorem 4.3]{BGST2023}). Moreover, the Lie algebra of the Grossman--Larson group $(G,*)$ is the Grossman--Larson Lie algebra $(\mathfrak g,\lgl-,-\rgl)$ \cite[Proposition 4.7]{BGST2023}.

\begin{prop}
Let $G$ be a post-Lie group with post-Lie algebra $\mathfrak g$. The post-Lie algebra of the opposite post-Lie group of $G$ is the opposite post-Lie algebra of $\mathfrak g$.
\end{prop}

\begin{proof}
It is well-known (and easily checked) that the Lie algebra of the opposite group is deduced from $\mathfrak g$ by changing the sign of the Lie bracket. Recalling the traditional notations
\[
	\mop{Ad}_g Y=\frac d{ds}\restr{t=0}g.\exp(sY).g^{-1},
\] 
and
$$
\mop{ad}_XY=\frac {d}{dt}\restr{t=0}\frac d{ds}\restr{s=0}\exp(tX).\exp(sY).\exp(-tX)=[X,Y],
$$
we now compute:
\begin{eqnarray*}
	\frac d{dt}\restr{t=0}\frac d{ds}\restr{s=0}\exp (tX)\blacktriangleright \exp (sY) 
	&=& \frac d{dt}\restr{t=0}\frac d{ds}\restr{s=0}\exp (tX).\big(\exp (tX)\rhd \exp (sY)\big).\exp (-tX)\\
	&=& \frac d{dt}\restr{t=0}\frac d{ds}\restr{s=0}\exp (tX).\exp s\big(\exp (tX)\rhd Y\big).\exp (-tX)\\
	&=&\frac d{dt}\restr{t=0}\mop{Ad}_{\exp(tX)}.\big(\exp (tX)\rhd Y\big)\\
	&=&(\mop{ad}_XX+L_X^\rhd)(Y)\\
	&=&[X,Y]+X\rhd Y=X\blacktriangleright Y,
\end{eqnarray*}
where, for any $g\in G$ and $Y\in\mathfrak g$, the notation $g\rhd Y$ stands for $\displaystyle \frac d{ds}\restr{s=0}\big(g\rhd\exp(sY)\big)$.
\end{proof}

\subsection{Post-Hopf algebras vs.~$D$-algebras}
\label{posthopf}

$D$-algebras were introduced by H.~Munthe-Kaas and W.~Wright in \cite[Definition 3]{MKW2008} -- independently of the introduction of post-Lie algebras by B.~Vallette \cite{BV2007}. 

\begin{defn} \cite{AEMM2022,MKW2008}
\label{def:Dalgebra}
The triple $(D,.,\rhd)$ consists of a unital associative algebra $(D,.)$ with product~$m_D(u \otimes v)= u.v$ and unit $\mathbf{1}$, carrying another product $\rhd: D \otimes D \rightarrow D$ such that $\mathbf{1} \rhd v =v$ for all $v \in D$. Let 
$$
	\mathfrak d (D)
	:=\{u \in D \ |\ u \rhd (v. w) 
	= (u \rhd v).w + v. (u \rhd w),\ \forall v,w \in D\}.
$$

We call $(D,.,\rhd)$ a $D$-algebra if the algebra product $.$ generates $D$  from $\{\mathbf{1},\mathfrak d (D)\}$ and furthermore for any $x \in \mathfrak d (D)$  and $v,w \in D$
\allowdisplaybreaks
\begin{align}
\label{D1}
	v \rhd x &\in \mathfrak d (D)\\
\label{D2}
	(x. v) \rhd w &= {\mathrm{a}}_{\rhd}(x,v,w),
\end{align}  
where the associator is defined by ${\mathrm{a}}_{\rhd}(x,v,w):=x\rhd(y\rhd w)-(x\rhd y)\rhd w$.
\end{defn} 

Although the enveloping algebra of a post-Lie algebra always carries a $D$-algebra structure, the converse is not necessarily true. The paradigmatic example given in \cite{MKW2008} is $\mathcal D:=C^\infty\big(\mathcal{M},\mathcal U(\mathfrak g)\big)$ where $\mathcal{M}$ is a smooth manifold, $G$ is a Lie group acting on $\mathcal{M}$, with Lie algebra $\mathfrak g$, and where $\mathcal U(\mathfrak g)$ is the universal enveloping algebra of $\mathfrak g$. The product is the pointwise product, and the action $u \rhd -$ is given by the differential operator on $\mathcal{M}$ defined by $u$ through the action of $G$ on $\mathcal{M}$. It is easily seen \cite{ELM2015} that $\mathcal L:=C^\infty(\mathcal{M},\mathfrak g)$ is the post-Lie algebra $\mathfrak d(\mathcal D)$, but $\mathcal D$ is a quotient of $\mathcal U(\mathcal L)$.\\

The $D$-algebra $\mathcal L$ carries an associative Grossman--Larson product $*$ corresponding to the composition of differential operators. It verifies
\begin{equation*}
	u\rhd(v\rhd w)=(u*v)\rhd w
\end{equation*}
for any $u,v,w\in\mathcal D$. It is compatible with the post-Lie algebra structure of $\mathcal L$ in the sense that, for any $X,Y\in\mathcal L$ we have $X*Y-Y*X=\lgl X,Y\rgl$. It is not clear whether any $D$-algebra admits such a Grossman--Larson product.\\

Post-Hopf algebras have been recently introduced in \cite{LST2022}. See also \cite{MQS2020}. They are examples of "good $D$-algebras", i.e.~they are always equipped with a Grossman--Larson product. 

\begin{defn}\cite[Definition 2.1]{LST2022}, \cite[Definition 5.1]{BGST2023}
A post-Hopf algebra $(H,.,\mathbf 1,\Delta,\varepsilon,S,\rhd)$ consists of a cocommutative Hopf algebra  $(H,.,\mathbf 1,\Delta,\varepsilon,S)$, where $\rhd:H\otimes H\to H$ is a coalgebra morphism such that, using Sweedler's notation $\Delta x=\sum_{(x)}x_1\otimes x_2$:
\begin{itemize}
\item for any $x,y,z\in H$ we have
\begin{equation*}
	x\rhd(y.z)=\sum_{(x)}(x_1\rhd y).(x_2\rhd z),
\end{equation*}
\item for any $x,y,z\in H$ we have
\begin{equation*}
	x\rhd(y\rhd z)=\sum_{(x)}\big(x_1.(x_2\rhd y)\big)\rhd z,
\end{equation*}
\item The operator $L^\rhd:H\to\mop{End} H$ defined by $L^\rhd_x(-):=x\rhd-$ admits an inverse $\beta^\rhd$ for the convolution product in $\mop{Hom}(H,\mop{End} H)$.
\end{itemize}
\end{defn}

From \cite[Theorem 2.4]{LST2022} (see also \cite[Theorem 5.3]{BGST2023}), any post-Hopf algebra $H$ admits a Grossman--Larson product $*$ and a linear map $S_*:H\to H$ making $\wt H:=(H,*,\mathbf 1,\Delta,\varepsilon, S_*)$ another cocommutative Hopf algebra. The Grossman--Larson product is given by
\begin{equation}
	x*y=\sum_{(x)}x_1.(x_2\rhd y),
\end{equation}
and the corresponding antipode is given by
\begin{equation}
	S_*(x)=\sum_{(x)}\beta^\rhd_{x_1}\big(S(x_2)\big).
\end{equation}
Remark 4 in \cite{AEMM2022} can be reformulated as follows: the universal enveloping algebra of a post-Lie algebra is a post-Hopf algebra.

\begin{rmk}\rm
Several results in \cite{AEMM2022} can be reformulated and showed in the post-Hopf algebra framework, with literally the same proofs. For example \cite [Theorem 2]{AEMM2022}, the product $.$ in a post-Hopf algebra can be expressed from the Grossman--Larson product $*$ and the corresponding antipode $S_*$ by the formula
\begin{equation*}
	x.y=\sum_{(x)}x_1*\big(S_*(x_2)\rhd y\big).
\end{equation*}
\end{rmk}

\begin{rmk}\rm
The notion of dual post-Hopf algebra can be introduced, by dualizing the axioms. This leads to a unital algebra carrying two different coproducts, yielding two Hopf algebras in cointeraction \cite{Manchon2018cointeraction}. A more detailed account will be given in forthcoming work. 
\end{rmk}

\section{Weak post-groups}\label{sect:weak}


\noindent The defining axioms of a post-group have been relaxed by S.~Wang \cite{W2023} as follows:

\begin{defn}\cite[Definition 2.1]{W2023}
A \textbf{(left) weak twisted post-group} is a quadruple $(G,.,\rhd,\phi)$ where the product~$.$ is a group structure, $\phi:G\to G$ is a map called \textbf{cocycle}, and $\rhd:G\times G\to G$ is a binary law such that 
\begin{itemize}
\item the left multiplication operators $L_a^\rhd=a\rhd -:G\to G$ are group morphisms,
\item the following identity holds $(a*b)\rhd c=a\rhd(b\rhd c)$ with 
$$
	a*b:=\phi(a).(a\rhd b),
$$
\item the following compatibility condition hods:
\[\phi(a*b)=a*\phi(b).\]
\end{itemize}
\end{defn}

A post-group is therefore a weak twisted post-group with trivial cocycle (i.e.~$\phi=\mop{Id}_G$) in which the $L_a^\rhd$'s are bijective.

\begin{prop}\label{wtpg}{\rm (from \cite[Theorem 2.7]{W2023})}
Let $(G,.,\rhd,\phi)$ be a weak twisted post-group. Then $(G,*)$ is a semigroup, and the binary product $\rhd$ is a left action of $(G,*)$ on the group $(G,.)$ by endomorphisms. If moreover $\phi=\mop{Id}_G$ and $L_e^\rhd$ is bijective, then $(G,*)$ is a monoid with unit $e$.
\end{prop}

\begin{proof}
Theorem 2.7 in \cite{W2023} is stated and proved for twisted post-groups. Proposition \ref{wtpg} states what survives from it for weak twisted post-groups. Associativity of $*$ is proved by direct computation:
\begin{eqnarray*}
	(a*b)*c
	&=&\big(\phi(a*b)\big).\big((a*b)\rhd c\big)\\
	&=&\big(a*\phi(b)\big).\big((a*b)\rhd c\big)\\
	&=&\Big(\phi(a).\big(a\rhd \phi(b)\big)\Big).\big(a\rhd(b\rhd c)\big)\\
\end{eqnarray*}
whereas
\begin{eqnarray*}
	a*(b*c)
	&=&a*\big(\phi(b).(b\rhd c)\big)\\
	&=&\phi(a).\Big(a\rhd\big(\phi(b).(b\rhd c)\big)\Big)\\
	&=&\phi(a).\Big(\big(a\rhd \phi(b)\big).\big(a\rhd(b\rhd c)\big)\Big).
\end{eqnarray*}
The Grossman--Larson semigroup $(G,*)$ acts on the group $(G,.)$ by endomorphisms as a direct consequence of the second axiom. Finally, if $\phi=\mop{Id}_G$ and if $L_e^\rhd$ is bijective, then we have $e\rhd a=a$ for any $a\in G$ (with the same proof as in the post-group case), therefore we have $e*a=e.(e\rhd a)=a$ and $a*e=a.(a\rhd e)=a$.
\end{proof}

\noindent From Proposition \ref{wtpg}, we shall adopt the following definition in the sequel:

\begin{defn}\label{wpg}
 A \textbf{weak post-group} is a weak twisted post-group with trivial cocycle and such that $L_e^\rhd$ is bijective.
\end{defn}

\section{Examples}
\label{examples}

\subsection{Two post-groups naturally associated to a group}

The \textbf{trivial post-group} associated to a group $(G,.)$ is given by $(G,.,\rhd)$ where $a\rhd b=b$ for any $a,b\in G$. All maps $L_a^{\rhd}:G\to G$ are therefore equal to the identity, and the Grossman--Larson product $*$ coincides with the group-law of the group $(G,.)$.\\

The \textbf{conjugation post-group} is the opposite post-group of the trivial post-group. It is given by $a\bcdot b:=b.a$ and $a\blacktriangleright b:= a.b.a^{.-1}$. The Grossman--Larson product is of course the same as the one of the trivial post-group, namely $a*b=a.b$.

\subsection{Weak post-groups from group actions}

Let $M$ be a set endowed with a right action $\rho: M \times G \to M$  of the $(G,.)$. We use the notation $ma$ for $\rho(m,a)$. Let $\mathcal G:=G^M$ be the set of maps from $M$ into $G$, endowed with the pointwise product
\begin{equation}
\label{produit-jauge}
	f.g(m):=f(m).g(m).
\end{equation}
Let us introduce the map $\rhd:\mathcal G\times\mathcal G\to\mathcal G$ defined by
\begin{equation}
\label{action-jauge}
	(f\rhd g)(m):=g\big(mf(m)\big).
\end{equation}

\begin{thm}
The triple $(\mathcal G,.,\rhd)$ is a weak post-group in the sense of Definition \ref{wpg}. The Grossman--Larson product is given by
\begin{equation}
	f*g(m)=f(m).g\big(mf(m)\big).
\end{equation}
\end{thm}

\begin{proof}
It is clear that $(\mathcal G,.)$ is a group. The unit is given by the constant function $\mathbf e$ such that $\mathbf e(m)=e$ for any $m\in M$, where $e$ is the unit of $G$. The inverse is given by $f^{.-1}(m):=f(m)^{.-1}$ for any $m\in M$. For any $f\in\mathcal G$, the map $L_f^\rhd=f\rhd-$ is a group morphism: indeed, we have for any $f,g,h\in\mathcal G$ and any $m\in M$:
\begin{eqnarray*}
	\big(f\rhd (g.h)\big)(m)
	&=&(g.h)\big(mf(m)\big)\\
	&=&g\big(mf(m)\big).h\big(mf(m)\big)\\
	&=&(f\rhd g)(m).(f\rhd h)(m)\\
	&=&\big((f\rhd g).(f\rhd h)\big)(m).
\end{eqnarray*}
We have 
$$
	f*g(m)=f.(f\rhd g)(m)=f(m).g\big(m.f(m)\big), 
$$
and we also have 
$$
	(f*g)\rhd h=f\rhd(g\rhd h),
$$ 
for any $f,g,h\in\mathcal G$, as shown by the computation below (where $m$ is any element of $M$):
\begin{eqnarray*}
	(f*g)\rhd h(m)
	&=&h\big(m(f*g)(m)\big)\\
	&=&h\Big\{m\Big(f(m).g\big(mf(m)\big)\Big)\Big\}
\end{eqnarray*}
whereas
\begin{eqnarray*}
	f\rhd(g\rhd h)(m)
	&=&(g\rhd h)\big(mf(m)\big)\\
	&=&h\Big(\big(mf(m)\big)g\big(mf(m)\big)\Big)\\
	&=&h\Big\{m\Big(f(m).g\big(mf(m)\big)\Big)\Big\}.
\end{eqnarray*}
The identity $L_{\mathbf e}^\rhd=\mop{Id}_{\mathcal G}$ is obvious, which ends up proving the claim. 
\end{proof}
\begin{rmk}
The weak post-group $\mathcal G$ generally fails to be a post-group, as the left multiplication operator $L_f^\rhd$ is not necessarily bijective\footnote{We thank Pierre Catoire for bringing this fact to our attention.}: for example, if $M=G$ with $\rho(m,a)=m.a$, let $f\in\mathcal G$ given by $f(m)=m^{-1}$. For any $g\in\mathcal G$, we have
\[(f\rhd g)(m)=g\big(m.f(m)\big)=g(m.m^{-1})=g(e),\]
hence the image of $L_f^\rhd$ only consists in constant functions.
\end{rmk}

\subsection{Weak post-Lie groups from Lie group actions}\label{MKW-PG}

Suppose now that $\mathcal{M}$ is a smooth manifold together with a smooth right action of a Lie group $G$ on it. We denote by $\mathcal G$ the set of smooth maps from $\mathcal{M}$ into $G$, and by $\mathcal L$ the vector space of smooth maps from $\mathcal{M}$ into $\mathfrak g$ (thus matching the notation of Paragraph \ref{posthopf}). Any $X \in \mathcal L$ gives rise to a vector field $\wt X$ on $\mathcal{M}$ defined by
\begin{equation}
	\wt X.\varphi(m):=\frac d{dt}\restr{t=0}\varphi\Big(m\exp\big(tX(m)\big)\Big).
\end{equation}

It is well-known that $\mathcal L$ is a post-Lie algebra \cite{ELM2015, HA13,MKW2008}. In fact, $\mathcal L$ is the post-Lie algebra of the "infinite-dimensional weak post-Lie group" $\mathcal G$. The action $\rhd: \mathcal L \times \mathcal L \to \mathcal L$ can indeed be derived from the action $\rhd: \mathcal G \times \mathcal G \to \mathcal G$ via differentiation as in \eqref{diff} as follows: for any $X,Y \in \mathcal L$ and $m \in \mathcal{M}$ we have
\begin{eqnarray*}
	\frac{d}{dt}\restr{t=0}\frac{d}{ds}\restr{s=0}\exp(tX)\rhd\exp(sY)(m)
	&=&\frac{d}{dt}\restr{t=0}\frac{d}{ds}\restr{s=0}\exp(sY)\Big(m\big(\exp (tX)(m)\big)\Big)\\
	&=&\frac{d}{dt}\restr{t=0}Y\Big(m\big(\exp (tX)(m)\big)\Big)\\
	&=&\wt X.Y(m)\\
	&=&(X\rhd Y)(m).
\end{eqnarray*}
The triple $(\mathcal L,[-,-],\rhd)$, where $[-,-]$ is the pointwise Lie bracket, with the product $\rhd$ appearing at the end of the computation above, is the post-Lie algebra described in \cite{HA13, ELM2015}.

\begin{rmk}\rm
Using the language of Lie algebroids, $(\mathcal L,[-,-])$ is the Lie algebra of sections of the Lie algebroid $\mathcal{M} \times \mathfrak g$ of the Lie group bundle $\mathcal{M} \times G$ above $\mathcal{M}$, whereas $(\mathcal L,\lgl -,-\rgl)$ is the Lie algebra of sections of the Lie algebroid of the transformation groupoid $\mathcal{M} \rtimes G$. The anchor map of the first Lie algebroid is trivial, the anchor map of the second is given by $X \mapsto \wt X$. For an account of Lie algebroids and Lie groupoids, we refer the reader to the textbook \cite{MacK2013}.
\end{rmk}

\subsection{The free post-group generated by a diagonal left-regular magma}

Let $\mathbf M$ be a left-regular magma, i.e.~a set endowed with a binary product $\rhd: \mathbf M \times \mathbf M \to \mathbf M$ such that the map $L_m^\rhd=m\rhd-$ is bijective for any $m\in\mathbf M$. Equivalently, a left-regular magma is a set $\mathbf M$ endowed with a left action $\wt\rhd$ of the free group $F_{\mathbf M}$ generated by $\mathbf M$. This action can in turn be uniquely extended to a left action of $F_{\mathbf M}$ on itself by group automorphisms. The left-regular magma $\mathbf M$ is called \textbf{diagonal} if moreover the map $\Lambda: \mathbf M \to \mathbf M$, $m \mapsto \Lambda(m)= (L_m^\rhd)^{ -1}(m)$ is a bijection.

\begin{thm}\label{main-free}
Let $\mathbf M$ be a diagonal left-regular magma. Denoting by a lower dot $.$ the multiplication of the free group $F_{\mathbf M}$, there is a unique way to extend the binary product $\rhd:\mathbf M\times \mathbf M\to\mathbf M$ to $F_{\mathbf M}$ such that $(F_{\mathbf M},.,\rhd)$ is a post-group. The post-group $F_{\mathbf M}$ thus obtained verifies the universal property making it the free post-group generated by the magma $\mathbf M$, namely for any post-group $(G,.,\rhd)$ and for any magma morphism $\varphi: \mathbf M\to (G,\rhd)$ there is a unique post-group morphism $\Phi:F_{\mathbf M}\to G$ making the diagram below commute.
\begin{equation}
\label{post-groupe-univ}
\xymatrix
	{\mathbf M\ar[dr]^\varphi\ar[d] &\\
	F_{\mathbf M}\ar[r]_\Phi & G
}
\end{equation}
\end{thm}

\begin{proof}
We have to construct a family of group automorphisms $L_a^\rhd:F_{\mathbf M}\to F_{\mathbf M}$ for any $a\in F_{\mathbf M}$ such that \eqref{GL} is verified for any $a,b,c\in F_{\mathbf M}$. In other words,
\begin{equation}
\label{GL-bis}
	L_a^\rhd\circ L_b^\rhd = L^\rhd_{a.(a\rhd b)},
\end{equation}
which can also be written as
\begin{equation}
\label{GL-ter}
	L^\rhd_{a.b} = L_a^\rhd\circ L_{a\rhdi b}^\rhd,
\end{equation}
where $a\rhdi -$ is a shorthand for the inverse group automorphism $(L_a^\rhd)^{-1}$. We construct the $L_a^\rhd$'s by induction on the length $|a|$ of the canonical word representation of $a$, and prove \eqref{GL-bis} by induction over $|a|+|b|$.\\

\noindent Let us first define the maps $L_a^\rhd$ when $a$ is of length zero or one. The maps $L_a^\rhd$ are given for $a\in \mathbf M$ by the magma structure, and \eqref{GL-bis} easily implies $L^\rhd_e=\mop{Id}_{F_{\mathbf M}}$. It remains to define $L^\rhd_{a^{.-1}}$ for $a\in\mathbf M$. From \eqref{GL-ter} with $b=a^{.-1}$ we get
\begin{equation}
\label{inverse}
	\mop{Id}_{F_{\mathbf M}}
	=L^\rhd_{a.a^{.-1}}
	=L_a^\rhd\circ L^\rhd_{a\rhdi a^{.-1}},
\end{equation}
which immediately yields
\begin{equation}
\label{inverse-bis}
	L^\rhd_{(a\rhdi a)^{.-1}}=(L_a^{\rhd})^{-1}.
\end{equation}
From the diagonality hypothesis on the left-regular magma $\mathbf M$, the map $\psi: \mathbf M \to \mathbf M^{.-1}$
\begin{equation}
\label{psi}
	\psi: a\mapsto (a\rhdi a)^{.-1}
\end{equation}
is a bijection. We can therefore define the group automorphisms $L_{a^{.-1}}^\rhd$ for any $a\in \mathbf M$ by
\begin{equation}
\label{inverse-def}
	L_{a^{.-1}}^\rhd:=\big(L^\rhd_{\psi^{-1}(a^{.-1})}\big)^{-1}.
\end{equation}

\begin{rmk}\label{involution}\rm
The map $\psi$ defined in \eqref{psi} can be seen as an involution on $\mathbf M\cup{\mathbf M}^{.-1}$. Indeed, from \eqref{inverse-def} the definition $\psi:a'\mapsto (a'\rhdi a')^{.-1}$ also makes sense for $a'\in {\mathbf M}^{.-1}$, and we get for any $a\in\mathbf M$:
\begin{eqnarray*}
	\psi^2(a)
	&=&\big(\psi(a)\rhdi \psi(a)\big)^{.-1}\\
	&=&\big((a\rhdi a)^{.-1}\rhdi (a\rhdi a)^{.-1}\big)^{.-1}\\
	&=&(a\rhdi a)^{.-1}\rhdi (a\rhdi a) \hskip 8mm \hbox{(from the group automorphism property)}\\
	&=&(a\rhdi a^{.-1})\rhdi (a\rhdi a) \hskip 8mm \hbox{(from the group automorphism property)}\\
	&=&\big((L^\rhd_{a\rhdi a^{.-1}})^{-1}\circ (L_a^\rhd)^{-1}\big)(a)\\
	&=&(L_a^\rhd\circ L^\rhd_{a\rhdi a^{.-1}})^{-1}(a)\\
	&=&a  \hskip 8mm \hbox{(from \eqref{inverse})}.
\end{eqnarray*}
Any $a'\in{\mathbf M}^{.-1}$ is uniquely written as $a'=\psi(a)$ with $a \in \mathbf M$. We have
$$
	\psi^2(a')=\psi^3(a)=\psi(a)=a',
$$
which terminate the proof of the statement, from which we easily get the mirror of \eqref{inverse}, namely
\begin{equation}
\label{inverse-ter}
	\mop{Id}_{F_{\mathbf M}}
	=L^\rhd_{a^{.-1}.a}
	=L_{a^{.-1}}^\rhd\circ L^\rhd_{a^{.-1}\rhdi a}.
\end{equation}
\end{rmk}

Let us now define $L_u^\rhd$ for any $u\in F_{\mathbf M}$. If $u$ is of length $n\ge 2$, it can be uniquely written $u=u'a$ with $a\in\mathbf M\cup{\mathbf M}^{.-1}$ and $u'\in F_{\mathbf M}$ of length $n-1$. From \eqref{GL-ter} we have $L^\rhd_u=L^\rhd_{u'}\circ L^\rhd_{u'\rhdi a}$. Iterating the process, we get the following expression for $L_u^\rhd$ in terms of the canonical representation $u=a_1\cdots a_n$ in letters $a_k\in \mathbf M\cup{\mathbf M}^{.-1}$:
\begin{equation}
\label{def-L}
	L_u^\rhd=L_{u_1}^\rhd\circ L^\rhd_{u_1\rhdi a_2}\circ\cdots\circ L^\rhd_{u_{n-1}\rhdi a_n},
\end{equation}
where $u_k$ is the prefix $a_1\cdots a_k$ for any $k=1,\ldots,n-1$. It remains to check the post-group property in full generality, namely $L^\rhd_{uv}=L^\rhd_u\circ L^\rhd_{u\rhdi v}$ for any $u,v\in\mathcal F_{\mathbf M}$. We proceed by induction on $|u|+|v|$. The cases with $|u|=0$ or $|v|=0$ are immediate. The case $|u|=|v|=1$ is covered by \eqref{def-L} when $u$ and $v$ are not mutually inverse and by \eqref{inverse} and \eqref{inverse-ter} when they are, which completely settles the case $|u|+|v|\le 2$. Suppose that the post-group property is true for any $u,v\in F_{\mathbf M}$ with $|u|+|v|\le n$, and choose $u,v$ with $|u|+|v|=n+1$. If $|v|=0$ there is nothing to prove, and if $|v|=1$ we are done by using \eqref{def-L}. Otherwise we have $v=v'.a$ with $|v'|=|v|-1$ and $a\in \mathbf M\cup\mathbf M^{.-1}$. We can now compute
\begin{eqnarray*}
	L^\rhd_{u.v}
	&=&L^\rhd_{u.v'.a}\\
	&=& L^\rhd_{u.v'}\circ L_{(u.v')\rhdi a} \hskip 8mm\hbox{by \eqref{def-L}}\\
	&=& L^\rhd_u\circ L^\rhd_{u\rhdi v'}\circ L^\rhd_{(L^\rhd_{uv'})^{-1}(a)}\hskip 8mm\hbox{by induction hypothesis}\\
	&=& L^\rhd_u\circ L^\rhd_{u\rhdi v'}\circ L^\rhd_{(L^\rhd_u\circ L^\rhd_{u\rhdi v'})^{-1}(a)}\hskip 8mm\hbox{by induction hypothesis}\\
	&=& L^\rhd_u\circ L^\rhd_{u\rhdi v'}\circ L^\rhd_{(u\rhdi v')\rhdi(u\rhdi a)}\\
	&=&  L^\rhd_u\circ L_{(u\rhdi v').(u\rhdi a)}^\rhd\hskip 8mm\hbox{by induction hypothesis}\\
	&=&L^\rhd_u\circ L^\rhd_{u\rhdi(v'.a)} \hskip 8mm\hbox{by the group automorphism property of $u\rhdi -$}\\
	&=&L^\rhd_u\circ L^\rhd_{u\rhdi v}.
\end{eqnarray*}
Finally, if $G$ is another post-group, any set map $\varphi:\mathbf M\to G$, a fortiori any magma morphism, admits a unique group morphism extension $\Phi:F_{\mathbf M}\to G$ making the diagram \eqref{post-groupe-univ} commute. Checking that $\Phi$ is a morphism of post-groups is left to the reader. 
\end{proof}

\subsection{A group isomorphism}

Given a post-group $(G,.,\rhd)$, the two groups $(G,.)$ and $(G,*)$ are in general not isomorphic: for example, in a typical pre-group, the first one is Abelian and the second one is not. We shall however see that the free post-group generated by left-regular diagonal magma provides an isomorphism between both associated groups. We keep the notations of the previous paragraph.

\begin{prop}\label{iso}
There is a unique group isomorphism
$$
	\mathcal J:(F_{\mathbf M},.)\longrightarrow  (F_{\mathbf M},*)
$$
such that $\mathcal J\restr{\mathbf M}=\mop{Id}_{\mathbf M}$.
\end{prop}

\begin{proof}
Existence and uniqueness of the group morphism $\mathcal J$ is immediate from the definition of a free group. It remains to show that $\mathcal J$ is an isomorphism. From \eqref{inverse-GL} we have $\mathcal J(a^{.-1})=(a\rhdi a)^{.-1}=\psi(a)^{.-1}$ for any $a\in\mathbf M$. We remark that $\mathcal J$ is therefore a bijection from $\mathbf M\cup{\mathbf M}^{.-1}$ onto itself. For any element $u=a_1\cdots a_n\in F_{\mathbf M}$ where the letters $a_k$ are in $\mathbf M\cup{\mathbf M}^{.-1}$, we have
\begin{equation*}
	\mathcal J(u)=a'_1*\cdots *a'_n
	=a'_1.(a'_1\rhd a'_2).\big(a'_1\rhd(a'_2\rhd a'_3)\big)\ldots (L^\rhd_{a'_1}\circ\cdots\circ L^\rhd_{a'_{n-1}})(a'_n)
\end{equation*}
with $a'_k=\mathcal J(a_k)$. In order to prove that any word $v:=b_1\cdots b_n$ can be written in the form $\mathcal J(u)$ with $u=a_1\cdots a_n$, we have to solve the triangular system
$$
	b_1=a'_1,
	\quad 
	b_2=a'_1\rhd a'_2,\ldots, 
	\quad 
	b_n= (L^\rhd_{a'_1}\circ\cdots\circ L^\rhd_{a'_{n-1}})(a'_n).
$$
The letters $a'_k$ are therefore recursively given by $a'_1=b_1$ and $a'_n=(L^\rhd_{a'_1}\circ\cdots\circ L^\rhd_{a'_{n-1}})^{-1}(b_n)$. The first ones are
$$
	a'_1=b_1, \ a'_2=b_1\rhdi b_2, \ a'_3
	=(b_1\rhdi b_2)\rhdi(b_1\rhdi b_3),\ a'_4
	=\big((b_1\rhdi b_2)\rhdi(b_1\rhdi b_3)\big)\rhdi
	\big((b_1\rhdi b_2)\rhdi(b_1\rhdi b_4)\big),\ldots,
$$
and the letters $a_k$ are then given by $a_k=\mathcal J^{-1}(a'_k)$. This builds up an inverse $\mathcal K$ for $\mathcal J:F_{\mathbf M}\to F_{\mathbf M}$ and therefore finishes the proof of Proposition \ref{iso}.
\end{proof}

\subsection{An analogy with Gavrilov's $K$-map}

Recall \cite{AEMM2022,Foissy2018} that the free Lie algebra $\mathcal L_{M}$ (over some ground field $\mathbf k$) generated by any magmatic $\mathbf k$-algebra $(M,\rhd)$ carries a unique post-Lie algebra structure such that the action $\rhd:\mathcal L_{M}\times \mathcal L_{M} \to \mathcal L_{M}$ extends the magmatic product of $M$. The universal enveloping algebra $\mathcal U(\mathcal L_{M})$ is the tensor algebra $T(M)$. As the former is a Hopf algebra with respect to the cocommutative unshuffle coproduct $\Delta$, the latter carries a cocommutative post-Hopf algebra structure.\\

Gavrilov introduced the so-called $K$-map, which is crucial for building higher-order covariant derivatives on a smooth manifold endowed with an affine connection \cite{Gavrilov2007, Gavrilov2008}. See also \cite{AEMM2022} for a more recent account extending Gavrilov's main results. It is a linear endomorphism of $T(M)$ recursively defined as follows: $K(x)=x$ for any $x \in M$ and 
\begin{equation}
\label{defK}
	K(x_1\cdots  x_{k+1}) 
	= x_1 . K(x_2 \cdots  x_{k+1}) - K\big(x_1\rhd(x_2\cdots  x_{k+1})\big),
\end{equation}
for any elements $x_1,\ldots, x_{k+1} \in M$ (note that for notational transparency we denote by a simple dot the concatenation product in $T(M)$). In particular, $K$ maps $x_1 . x_2 \in T(M)$ to
$$	
	K(x_1 . x_2) = x_1 . x_2 - x_1 \rhd x_2
$$
and for $x_1 . x_2.  x_3 \in T(M)$ we obtain
\begin{eqnarray*}
	\lefteqn{K(x_1 . x_2.  x_3)
	=x_1.  K(x_2. x_3)-K \big(x_1\rhd(x_2 . x_3)\big)}\\
	&=&x_1. x_2.  x_3 - x_1. (x_2 \rhd x_3) - (x_1 \rhd x_2).  x_3 - x_2 . (x_1 \rhd x_3)
	 + x_2 \rhd(x_1 \rhd x_3) + (x_1 \rhd x_2) \rhd x_3.
\end{eqnarray*}
In the second equality we used that $x_1\rhd(x_2 . x_3) = (x_1\rhd x_2) . x_3 + x_2 . (x_1\rhd x_3)$. The map $K$ is clearly invertible, as $K(U)-U$ is a linear combination of terms of strictly smaller length than the length of the element $U \in T(M)$. The inverse $K^{-1}$ admits a closed formula in terms of set partitions (with blocks ordered according to their respective maximal elements) \cite[Paragraph 2.5.2]{AEMM2022}. Defining the Grossman--Larson product on the cocommutative post-Hopf algebra $T(M)$
\begin{equation}
\label{GLmagma}
	A * B = A_1 . (A_2 \rhd B),
\end{equation}
for $A,B \in T(M)$, it follows from \cite[Proposition 1]{Gavrilov2012} and \cite[Theorem 3]{AEMM2022} that $K$ is a Hopf algebra isomorphism from $\wt{H}=\big(T(M),*,\mathbf 1,\Delta,S_*,\varepsilon\big)$ onto $\big(T(M),.,\mathbf 1,\Delta,S,\varepsilon\big)$, that is, for $A,B \in T(M)$, we have
\begin{equation}
\label{Kmorphism}
	K(A * B)=K(A).K(B).
\end{equation}

The map $K$ admits a unique extension by continuity to the completion $\wh{T(M)}$ of $T(M)$ with respect to the grading. The set $G$ of group-like elements in $\wh{T(M)}$ is a post-group, and the restriction $K\restr G$ is a group isomorphism from $(G,*)$ onto $(G,.)$. The group isomorphism 
$$
	\mathcal K=\mathcal J^{-1}:(F_{\mathbf M},*)\to  (F_{\mathbf M},.)
$$ 
defined in the previous paragraph can therefore be seen as a group-theoretical version of Gavrilov's $K$-map.

Recall J. H. C.~Whitehead's definition of crossed group morphisms \cite{W1949}.

\begin{defn} \cite{W1949} \label{def:crossedmorph}
Let $(G,.)$ be a group. Suppose that the group $\Gamma$ acts on $G$ by automorphism. The action is denoted by $\rhd: \Gamma \times G \to G$. A crossed homomorphism $\phi: \Gamma \to G$ is a map such that for any $h_1,h_2 \in \Gamma$
\begin{equation}
\label{CrossedMorphism}
	\phi(h_1 h_2)=\phi(h_1) . (h_1 \rhd \phi(h_2)).
\end{equation}
\end{defn}

\begin{prop}\label{prop:postgroup}
Let $(G,.,\rhd)$ be a post-group with Grossman--Larson product $*$. The identity map $\id: (G,*) \to (G,.)$ is a crossed homomorphism. 
\end{prop}

\begin{proof}
The statement follows directly from the definition of the Grossman--Larson product (Definition~\ref{def:gl}).
\end{proof}

For the post-group $(G,.,\rhd)$ associated to the free post-Lie algebra generated by a magmatic algebra $(M,\rhd)$, Gavrilov showed in \cite[page 87]{Gavrilov2012} that $K^{-1}$ is a crossed morphism from the post-group $(G,.,\rhd)$ into itself in the sense of Whitehead. Indeed, $K$ is a group isomorphism between $(G,*)$ and $(G,.)$ \cite{AEMM2022}.  Hence, $K^{-1} = \id \circ K^{-1} : (G,.) \to (G,.)$ is a crossed morphism for the action 
$$
	g_1 \blacktriangleright g_2:= K^{-1}(g_1) \rhd g_2.
$$ 
of $(G,.)$ on itself. We then have
\begin{align*}
	K^{-1}(g_1 . g_2) 
	&= K^{-1}(g_1) * K^{-1}(g_2) \\
	&= K^{-1}(g_1) . \big( K^{-1}(g_1) \rhd K^{-1}(g_2) \big)\\
	&= K^{-1}(g_1) . \big( g_1 \blacktriangleright K^{-1}(g_2) \big).
\end{align*}

\begin{rmk}\rm
We close the paper by pointing at an interesting observation regarding Gavrilov's $K$-map extended to $\wh{T(M)}$ and flow equations on the post-group $G$. Consider the initial value problem
$$
	\frac{d}{dt} \exp^.(tx) = \exp^.(tx) . x 
$$
in $\wh{T(M)}$. Using \eqref{GLmagma} and the fact that $T(M)$ is a cocommutative post-Hopf algebra, the righthand side can be written
$$
	\exp^.(tx) . x = \exp^.(tx) * ( S_*(\exp^.(tx)) \rhd x), 
$$
because $\exp^.(tx)$ is a group-like and $S_*$ is the antipode of $\wt{H}$. Introducing 
\begin{equation}
\alpha(tx):=S_*\big(\exp^.(tx)\big)\rhd x,
\end{equation}
we see that
\begin{equation}
\label{Kde}
	\frac{d}{dt}K(\exp^.(tx))=K\big(\exp^.(tx).x\big)=K\big(\exp^.(tx)\big). \alpha(tx),
\end{equation}
\eject
where we used that $K\big(\alpha(tx)\big) = \alpha(tx)$ as $\alpha(tx)\in M$. The solution of the initial value problem \eqref{Kde} (the initial value being $K(1)=0$) can be expressed in terms of the right-sided Magnus expansion\footnote{Recall from \cite{M1954} that the Magnus expansion is the formal expression of the solution of initial value problem $\dot X=AX$, $X(0)=1$, in a noncommutative algebra. It is defined by $X(t)=\exp\big(\Omega(A)(t)\big)$, where $\Omega$ is recursively given by
\[
	\Omega(A)(t)=\int_0^t \sum_{k\ge 0}\frac{B_k}{k!} (\mop{ad}^k_{\Omega(A)}A)(u)\,du.
\]
The right-sided Magnus expansion is the solution of initial value problem $\dot X=XA$, $X(0)=1$. It is given by the same expression except that the Bernoulli number $B_k$ are replaced by the modified Bernoulli numbers $\widetilde B_k$, where $\wt B_1=-B_1=1/2$ and $\wt B_k=B_k$ for $k\neq 1$.} $\Omega\big(\alpha(tx)\big) \in \wh{\mop{Lie} (M)}$ as
$$
	K\big(\exp^.(tx)\big) = \exp^.\big(\Omega(\alpha(tx))\big),
$$
which implies
$$
K\big(\exp^.(tx)\big) 
=\exp^.\Big(K\circ K^{-1}\big(\Omega(\alpha(tx))\big)\Big)
=K\Big(\exp^* K^{-1}\big(\Omega(\alpha(tx))\big)\Big)
=K\Big(\exp^* \big(\Omega_*(\alpha(tx))\big)\Big),
$$ 
where we used that 
$$
	K^{-1}[A,B]
	= K^{-1}(A) * K^{-1}(B) - K^{-1}(B) * K^{-1}(A)
	=[\![K^{-1}(A),K^{-1}(B)]\!].
$$
The last equality defines the Grossman--Larson Lie bracket and $\Omega_*(\alpha(tx))$ is the right-sided Magnus expansion defined in terms of the Grossman--Larson Lie bracket, recursively given by
$$
	\Omega_*(\alpha(tx)) = \int_0^t \sum_{n \ge 0} \frac{\wt B_n}{n!} (\text{ad}^{*n}_{\Omega_*(\alpha)}\alpha)(ux)\, du,
$$
where for $n > 0$, we define $\text{ad}^{*n}_a(b):=[\![a, \text{ad}^{*n-1}_a(b)]\!]$ and $\text{ad}^{*0}_a(b):=b$. The $\wt B_n$'s are the modified Bernoulli numbers $1,1/2,1/6,0,-1/30,0,1/42,\ldots$ Moreover, a simple computation gives 
\begin{align*}
	\frac{d}{dt} \alpha(tx) 
	&= S_*\Big(\frac{d}{dt} \exp^.(tx) \Big) \rhd x \\
	&= S_*\big(\exp^.(tx) * ( S_*(\exp^.(tx)) \rhd x)\big) \rhd x \\
	&= \Big (S_*\big( S_*(\exp^.(tx)) \rhd x \big) * S_*(\exp^.(tx)) \Big) \rhd x \\
	&= \Big( -\big(S_*(\exp^.(tx)) \rhd x\big) * S_*(\exp^.(tx)) \Big) \rhd x \\
	&=  -\big(S_*(\exp^.(tx)) \rhd x\big) \rhd \big( S_*(\exp^.(tx)) \rhd x \big) \\
	&= - \alpha(tx) \rhd \alpha(tx). 
\end{align*}
Compare with Gavrilov \cite[Lemma 1 \& Lemma 2]{Gavrilov2007} as well as reference \cite{KMMK2017Rmatrix}. 
\end{rmk}


\end{document}